\documentclass[12pt]{amsart}     

\usepackage{amsthm, amssymb, amsmath} 
\usepackage{graphicx, caption, float}
\usepackage{hyperref}

\usepackage[margin=2cm]{geometry}

\theoremstyle{plain}
\newtheorem{theorem}{Theorem}[section] 
\newtheorem{prop}[theorem]{Proposition}
\newtheorem{corollary}[theorem]{Corollary}
\newtheorem{algorithm}[theorem]{Algorithm}

\title{Rules for Folding Polyminoes from One Level to Two Levels}

\author{Julia Martin}
\address{Julia Martin, Mathematics Department, Oswego State University of New York, Oswego, New York}
\email{\href{mailto:jmartin7@oswego.edu}{jmartin7@oswego.edu}}

\author{Elizabeth Wilcox}
\address{Elizabeth Wilcox, Mathematics Department, Oswego State University of New York, Oswego, New York}
\email{\href{mailto:elizabeth.wilcox@oswego.edu}{elizabeth.wilco@oswego.edu}}

\dedicatory{Dedicated to Lunch Clubbers Mark Elmer, Scott Preston, Amy Hannahan, and Max Robertson}

\begin{document}

\begin{abstract} Polyominoes have been the focus of many recreational and research investigations.  In this article, the authors investigate whether a paper cutout of a polyomino can be folded to produce a second polyomino in the same shape as the original, but now with two layers of paper.  For the folding, only ``corner folds'' and ``half edge cuts'' are allowed, unless the polyomino forms a closed loop, in which case one is allowed to completely cut two squares in the polyomino apart.  With this set of allowable moves, the authors present algorithms for folding different types of polyominoes and prove that certain polyominoes can successfully be folded to two layers.  The authors also establish that other polyominoes cannot be folded to two layers if only these moves are allowed. {\bf Keywords:} Polyominoes (MSC2010 05B50), Recreational Mathematics (MSC2010 00A08).\end{abstract}

\maketitle


A {\it polyomino} is a geometric figure formed by joining one or more squares of equal side length together, edge-to-edge.  The most familiar polyominoes are dominoes, formed by joining two squares together, and the tetrominoes used in the popular game Tetris.  Researchers (including those engaging in mathematical recreation!) have enjoyed studying polyominoes in relation to tiling the plane \cite{Golomb, Martin}, game play in Tetris \cite{brzu1992, burgi1997}, dissecting geometric figures \cite{Frederickson}, and a whole host of other fascinating problems.  

This current investigation addresses the question of how to fold a paper cut-out of a polyomino, using prescribed allowable folds, so that the resulting shape is exactly the same as the original but now with exactly two layers (``levels'') of paper over the entirety of the polyomino.  If one can succeed in this endeavor with a particular polyomino, it is said that the polyomino can be {\it folded from one level to two levels}, or {\it folded} for short.  This question was originally brought up in Frederickson's ``Folding Polyominoes from One Level to Two'' \cite{fred2011} from 2011.  Frederickson showed several solutions and introduced readers to the terminology of the field; here the authors give explicit algorithms for how to fold certain polyominoes and establish some of the theory dictating which polyominoes can be folded from one level to two levels.

In the language introduced by Frederickson, a {\it HV-square} is a single square in a polyomino that is attached to at least one square along a vertical edge and to at least one square along a horizontal edge. A {\it well-formed polyomino} has no adjacent HV-squares. By contrast, a {\it non-well-formed polyomino} has two or more adjacent HV-squares and a polyomino with no HV-squares at all is called a {\it chain polyomino}. A chain of squares attached to a HV-square is called an {\it appendage}.  These concepts are illustrated in Figure~\ref{vocab}.  The number of squares in a chain polyomino is referred to as its {\it length}.
\begin{figure}[H]
\centering
\includegraphics[width=5.5in]{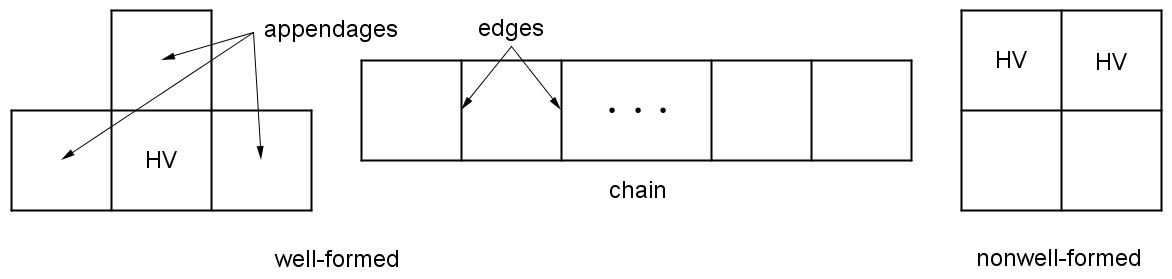}
\caption{}\label{vocab}
\end{figure}
The word {\it genus} refers to the number of holes in a polyomino, following the usage by topologists.  All of the polyominoes shown in Figure~\ref{vocab} are of genus 0 and the folding of such well-formed polyominoes is addressed in Section~\ref{intro}.  Polyominoes of larger genus are addressed in Section~\ref{largergenus}.

Frederickson and others were particularly interested in the figures that arise from ``dissecting'' a polyomino at the fold creases.  These figures are called {\it dissections}.  In this paper, we refer specifically to four main types of dissections:  corner triangle, middle square with adjacent corner triangle, skewed parallelogram, and right triangle. Figure~\ref{wfp} illustrates these dissections.

\begin{figure}[H]
\centering
\includegraphics[width=3.5in]{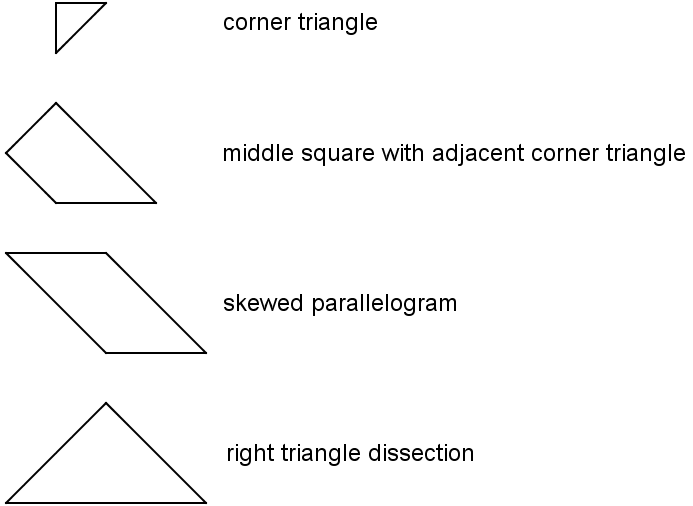}
\caption{}\label{wfp}
\end{figure}

From the dissections shown in Figure~\ref{wfp}, one can deduce the type of fold that we employ.  A {\it corner fold} can be performed on any corner of a square. Mark the midpoints of two consecutive sides and fold the corner of the square inwards, towards the center of the square.  Corner folds can be extended across an edge to a second adjacent square.  The reader may find it easiest to follow along with paper if after every fold, the polyomino is reoriented so that the next fold is created by {\it folding away}.  

{\it Half edge cuts} are also allowed when working with well-formed polyominoes, and in fact necessary when not dealing with chain polyominoes.  In this move, one cuts halfway along one edge of a square in the polyomino.  Remember that this is literally folding paper cutouts of polyominoes so performing half edge cuts in the middle of a block of HV-squares is not reasonable.  Additionally, the only time we cut an entire edge is when folding a polyomino of genus greater than 0.


\section{Well-Formed Polyominoes of Genus 0}\label{intro}

The easiest polyomino to fold is, without contest, a chain polyomino of any length.

\begin{prop} \label{propchain} Let $n$ be an integer, at least 2.  A chain polyomino of length $n$ can be folded from one level to two.  Moreover, the dissections will be $4$ corner triangles, $2$ middle squares with adjacent corner triangle, and $n-2$ skewed parallelograms. \end{prop}

\begin{proof} We want to show that chain polyominoes can be folded using three of the shape dissections of well-formed polyominoes seen in Figure~\ref{wfp}. Let a chain polyomino of length $n$ be given, as in Figure~\ref{chain}, with the squares labeled from 1 to $n$. 
\begin{figure}[H]
\centering
\includegraphics[width=3.75in]{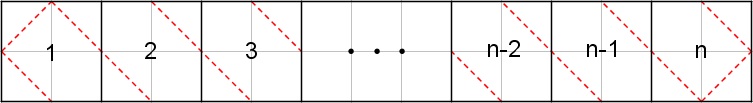}
\caption{chain of length $n$}\label{chain}
\end{figure}
Then in square 1, make two corner folds. Make a fold across the edge between square 1 and square 2 that is parallel to one of the corner folds. Repeat that fold between square $i$ and square $(i+1)$ until the folds cross into square $n$. Finish with the last two corner folds in square $n$. Therefore, by using four corner triangle dissections, two middle square with adjacent corner triangle dissections, and $n-2$ skewed parallelogram dissections, a chain polyomino of length $n$ can be folded from one level to two levels. \end{proof}

Note that there are two distinct ways to fold a chain polyomino because the ``folder'' makes a choice after the first step - there are two possible ways to make create a fold across square 1 and square 2 that is parallel to one of the corner folds in square 1.
\begin{figure}[H]
\centering
\includegraphics[width=4in]{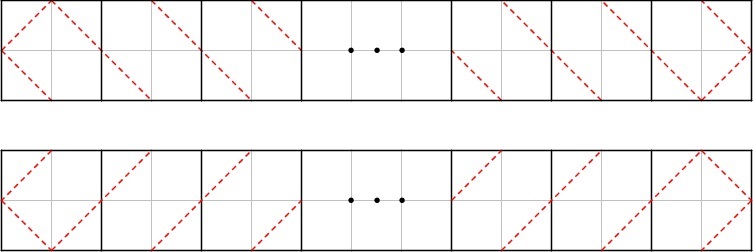}
\caption{alternative folds for chain}\label{2chains}
\end{figure}
%


\begin{prop} \label{propobs} Let $n$ be an integer, at least 2.  For a chain of length $n$, two perpendicular folds in any one of squares $2$ through $(n-1)$ (i.e., a right triangle dissection) is an obstruction to folding the polyomino from one level to two. 
\end{prop}

For example, suppose a chain of length $n$ is folded with two perpendicular folds in one of the squares besides square $1$ and $n$, as in Figure~\ref{obs}.
\begin{figure}[H]
\centering
\includegraphics[width=3.75in]{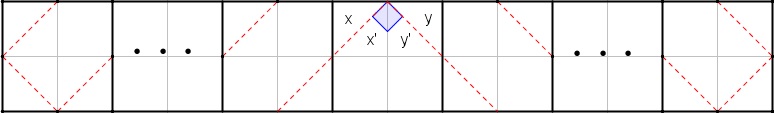}
\caption{two perpendicular folds in a middle square}\label{obs}
\end{figure}

\noindent
The two-level polyomino is not a chain because the resulting square overlaying the right triangle dissection will be an HV-square. Hence, two perpendicular folds in any one of squares $2$ through $(n-1)$ is an obstruction to folding a chain polyomino from one level to two levels.

\begin{corollary} \label{corobs} To introduce an HV-square with two appendages while folding a chain from one level to two, make two perpendicular corner folds across the edges of a square (i.e., make a right triangle dissection) that is not at the end of a chain. \end{corollary} 

Because there are no other types of folds allowed, the only possible ways to fold squares in the middle of a chain are using parallel folds across the edges (giving skew parallelogram dissections) or making two perpendicular corner folds (giving a right triangle dissection).  Thus Proposition~\ref{propobs} tells us that the only way a chain can be folded is the ways described by Proposition~\ref{propchain}, yielding the first theorem of this section:

\begin{theorem} \label{thmobs} There are exactly two ways to fold a chain polyomino from one level to two levels. \end{theorem}


In a well-formed polyomino, there are three possible configurations for a HV-square:  two appendages, three appendages, or four appendages.  These configurations are referred to as an $L$ shape, a $T$ shape, and an $X$ shape, respectively.  We now provide algorithms to successfully fold a polyomino with one of these configurations from one level to two.  Before starting, it's important to notice that by rotating a polyomino, with the HV-square as the center of rotation, there is essentially one orientation for each configuration.

\begin{algorithm}[\textbf{$L$ shape}] \label{algl} Turn the polyomino so there is a right horizontal appendage and an upward vertical appendage. Cut halfway along the edge of the HV-square and the vertical appendage from the right. Make a corner fold in the HV-square. Make a fold across the edge between the HV-square and the horizontal appendage that is perpendicular to the corner fold in the HV-square; from there, fold as needed down the appendage. Fold a corner fold across the square in the vertical appendage  directly next to the HV-square. This corner fold should be parallel to the corner fold in the HV-square. From there, fold as needed down the appendage.
\end{algorithm}
\begin{figure}[H]
\centering
\includegraphics[width=1.5in]{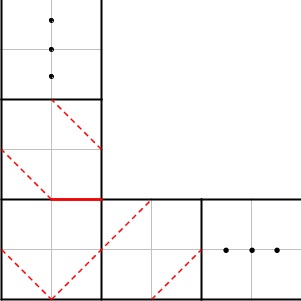}
\caption{cuts and folds for an $L$ shape}\label{L}
\end{figure}

\begin{algorithm}[\textbf{$T$ shape}] \label{algt} Orient as in Figure~\ref{T}. Make two cuts: first cut halfway along the edge of the HV-square and the upper vertical appendage from the right, then make another cut halfway along the edge between the HV-square and the right horizontal appendage from below. Fold across the edge between the HV-square and the square below, proceed to fold down the appendage as needed. As with the $L$ shape, make a corner fold on the horizontal appendage square that, if extended into the HV-square, would be perpendicular to the HV-square corner fold. Fold the upper vertical appendage as in $L$ shape algorithm. \end{algorithm}
\begin{figure}[H]
\centering
\includegraphics[width=1.4in]{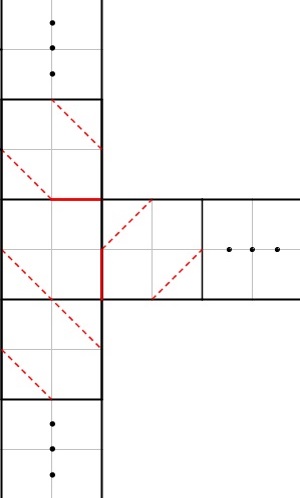}
\caption{cuts and folds for a $T$ shape}\label{T}
\end{figure}
\begin{algorithm}[\textbf{$X$ shape}] \label{algx} Cut halfway along the edge between the HV-square and the upward vertical appendage from the right, turn the polyomino 90$^\circ$ and repeat the cut; do this twice more until there is a cut between the HV-square and all appendages. Then make a corner fold in the square directly next to the HV-square in the vertical appendage, but do not fold into the HV-square. Proceed down the appendage as needed. Rotate the polyomino 90$^\circ$ and repeat the same fold three more times in each appendage and fold down each appendage as needed.
\end{algorithm}
\begin{figure}[H]
\centering
\includegraphics[width=2.5in]{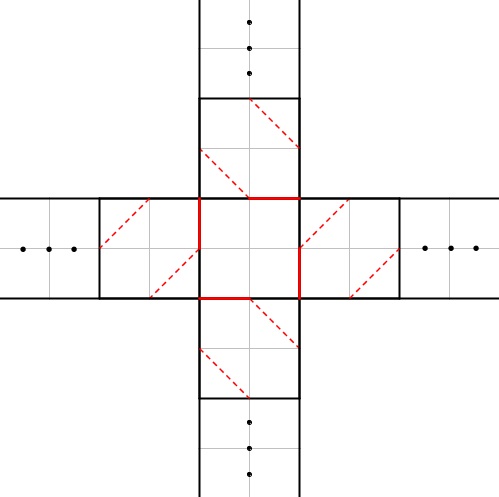}
\caption{cuts and folds for an $X$ shape}\label{X}
\end{figure}

An $X$ shape with the appendages not continuing beyond a single square has also been referred to as a {\it Greek Cross} (see \cite[p. 266]{fred2011}).  

These three algorithms provide one way to fold well-formed polyominoes with genus $0$ and one HV-square. However each algorithm could be performed from the perspective of a ``mirror image". For example, one could cut and fold the $L$ shape in Figure~\ref{L} according to a diagram in Figure~\ref{rL}, which is simply the original algorithm but reflected across the 45$^\circ$ diagonal of the HV-square. 
\begin{figure}[H]
\centering
\includegraphics[width=1.5in]{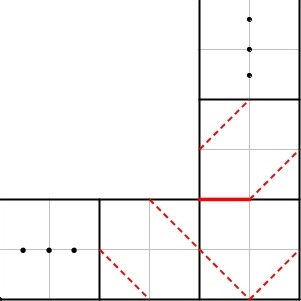}
\caption{alternative cuts and folds for an $L$ shape}\label{rL}
\end{figure}


When the polyomino is well-formed, there is at least one square between each pair of consecutive HV-squares and this provides enough space for the algorithms to be performed consecutively. The question is whether an appendage extending to the left will continue to left after folding, or if it will now extend to the right of its HV-square. The algorithms for HV-squares all preserve the orientation of the appendages relative to one another, and do not produce a ``mirror image'' of the original HV-square.  However, the folds and cuts for the HV-square are not the only orientation-changing maneuvers.

The squares between a pair of consecutive HV-squares are essentially a chain and are folded using the method of Proposition~\ref{propchain}.  Each corner fold that extends across adjacent squares flips the appendage $180^{\circ}$.  If there is an odd number of squares between consecutive HV-squares then there will be an even number of such folds, for no net change in orientation.  If there is an even number of squares between consecutive HV-squares then there is an odd number of folds across adjacent squares, and the orientation at the second HV-square will be incorrect.  Since the algorithms presented here for folding $X$ shapes, $L$ shapes, and $T$ shapes {\it all} preserve orientation, there's no way to correctly fold a polyomino with an even number of squares between a pair of consecutive HV-squares using only corner folds and half edge cuts.


Other algorithms may exist to fold the three configurations for HV-squares in well-formed polyominoes, leading to a solution to the problem of an even number of squares between a pair of consecutive HV-squares. However, these algorithms do lead to the conclusion:

\begin{theorem} \label{thmgenus0} Any well-formed polyomino of genus 0 with an odd number of squares between each pair of consecutive HV-squares can be folded from one level to two levels with only corner folds and half edge cuts.  A well-formed polyomino of genus 0 with an even number of squares between any pair of consecutive HV-squares cannot be folded from one level to two with only corner folds and half edge cuts.
\end{theorem}


\section{Well-Formed Polyominoes of Genus $>$ 0}\label{largergenus}

\subsection{Genus At Least 1}

\begin{prop}\label{propgenus1} Let a closed polyomnio have no appendages and genus 1 with a rectangular hole. If each side has odd length then the polyomino can be folded to 2 levels.
\end{prop}

\begin{figure}[H]
\centering
\includegraphics[width=2.75in]{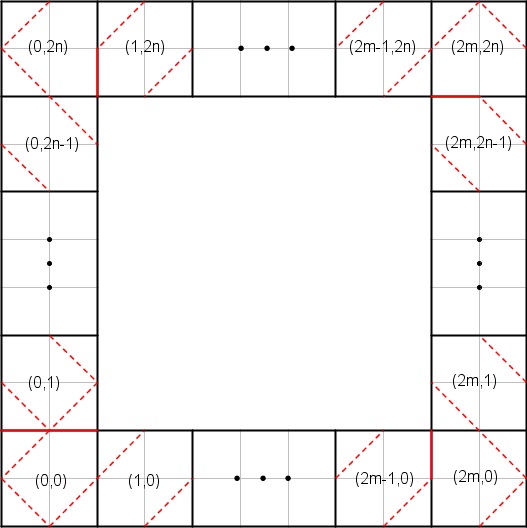}
\caption{}\label{genus1}
\end{figure}

\begin{proof} First, cut the edge between square $(0,0)$ and square $(0,1)$ in Figure~\ref{genus1}. Make two corner folds on square $(0,0)$. Fold across into square (1,0) parallel to the corner fold adjacent to square $(0,1)$. Fold across square $(1,0)$ through square $(2m-1,0)$ using the chain algorithm. Notice that when you fold as a chain, at the end of the appendage you are left with a half-square and a HV-square in the perfect position for the L-shape algorithm to be applied.
Use the L-shape algorithm to turn the corner. Repeat this procedure until the polyomino is completely folded.

\textit{Why is it that the 2 ends of the folded polyomino meet?} It has to do with the folding on the squares $(2m,y)$ where $0 \leq y \leq 2n$ and squares $(x,2n)$ where $ 1 \leq x \leq 2m$. The squares $(2m,x)$ where $1 \leq x \leq 2m$ form a chain of length $2n+1$ with both square $(2m,2n)$ and square $(2m,0)$ being HV-squares. There will be $2n-1$ complete folds across adjacent squares in the chain; in square $(2m,2n-1)$ we have a corner fold that also flips the appendage $180^{\circ}$. So there are a total of $2n$ folds, each flipping the appendage $180^{\circ}$ for no net change in orientation. Then the two appendages on the HV-squares will be parallel but in the same direction. The same will happen for squares $(x,2n)$ since this is also a chain of odd length. Then the two end cuts will come back together to create a polyomino of genus 1 and 2 levels. \end{proof}

\begin{corollary} \label{corodd} Since the sides of a $k \times j$ polyomino of genus 1 will have $k-1$ folds and $j-1$ folds on respective sides, the only way orientation can be preserved is if $k$ and $j$ are both odd. \end{corollary}

The folding algorithm does not require that the cut be made where specified; that designation was made only for consistency.  However, the folding algorithm truly {\it does} require the polyomino have an odd number of squares between each consecutive pair of HV-squares.  Consider a well-formed polyomino of genus 1 where the sides are $4 \times 3$, as in Figure~\ref{counterex}.
\begin{figure}[H]
\centering
\includegraphics[width=4in]{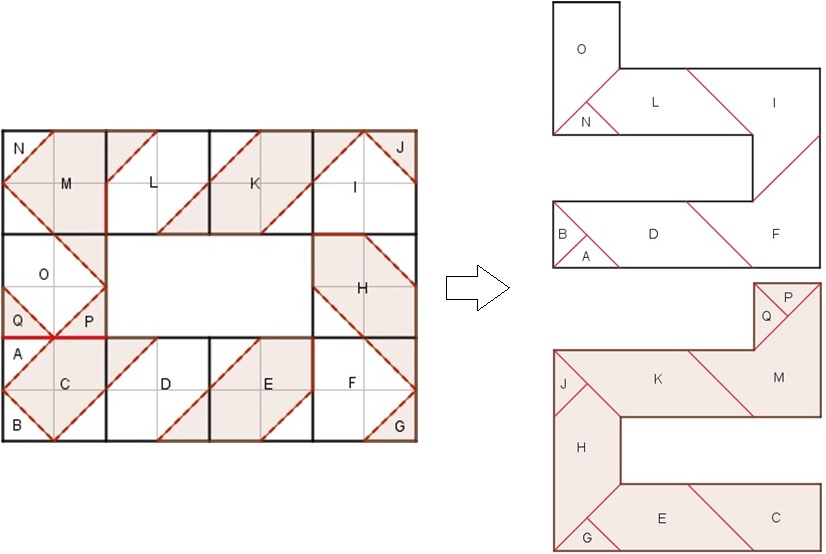}
\caption{}\label{counterex}
\end{figure}
If one orients the polyomino to begin by folding along a side with three squares, the folding seems to work -- it's only when one ``turns the corner'' and starts to work on a side with four squares that trouble strikes.  In Figure~\ref{counterex} we show how the folding will result if the algorithm starts along a side with four squares.  

Having observed that the algorithms created completely fail when there is an even number of squares between pairs of consecutive HV-squares, one might conclude the following: {\it A closed polyomino with no appendages, genus 1, and rectangular hole can be folded to 2 levels if and only if each side has odd length.}

In fact, the algorithm described in Proposition~\ref{propgenus1} is even broader; no part of the proof required that the well-formed polyomino be of genus 1, other than simply dictating a single cut.  In a closed polyomino of genus $g > 0$, one needs to make $g$ cuts, one to connect each ``hole'' to the ``outside'' and then the algorithms for folding do the rest of the work.

\begin{theorem}\label{thmnotrect} A well-formed closed polyomino with no appendages and an odd number of squares between each pair of consecutive HV-squares can be folded from one level to two.\end{theorem}

\subsection{Another Look at Genus 1}

While only allowing corner folds and half-edge cuts, the authors stumbled upon an alternative method for folding certain closed polyominoes of genus 1 with no appendages from one level to two levels.  This began with the following observation:

\begin{prop} \label{LtoChain}
A 1-level $L$ shape polyomino can be folded into a 2-level chain polyomino. \end{prop}

\begin{proof} Orient the $L$ shape polyomino as in the $L$ shape algorithm and fold the bottom left corner of the HV-square using a corner fold as in Figure~\ref{obs2}.  Fold the bottom left corner in the square above the HV-square; follow the fold into the HV-square top right corner and the square adjacent to the right bottom left corner. Make two more parallel folds on each appendage, then finish the appendages by following the chain algorithm. \end{proof}

\begin{figure}[H]
\centering
\includegraphics[width=2.3in]{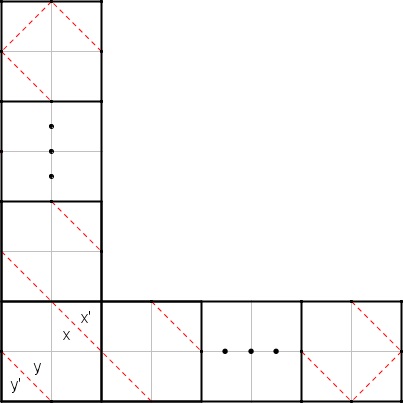}
\caption{}\label{obs2}
\end{figure}

In the situation of a well-formed polyomino of genus 0, the result of Proposition~\ref{LtoChain} is not particularly useful.  However, when considering well-formed polyominoes of genus 1 with a rectangular hole, the procedure described in Proposition~\ref{LtoChain} suddenly opens a door for folding without the use of any cuts.

\begin{prop} \label{propoddobs}
There is an alternative way of folding a $(2n+1) \times (2m+1)$ well-formed polyomino of genus 1 without needing any cuts. \end{prop}

\begin{figure}[H]
\centering
\includegraphics[width=3.5in]{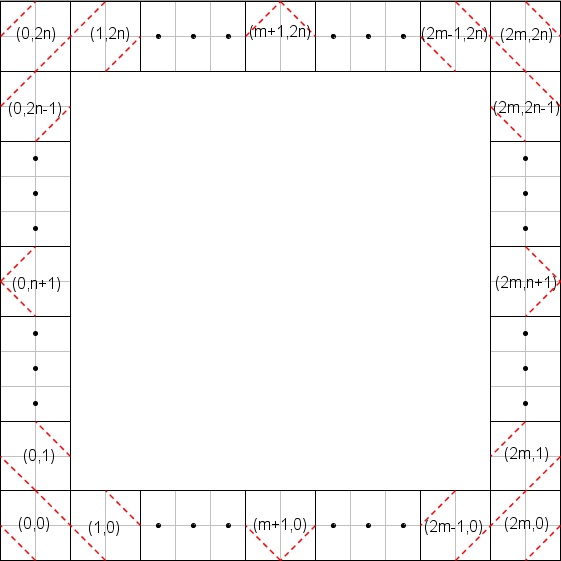}
\caption{}\label{obsfolds}
\end{figure}

\begin{proof} Begin on the middle square of all sides of the polyomino and make perpendicular folds as shown in Proposition~\ref{propobs}; this results in a perpendicular fold in the center of squares $(m+1,0)$, $(0, n+1)$, $(2m, n+1)$, and $(m+1,2n)$. Ensure the vertex of the corner is on the outside edge of each square. Taking a closer look at the polyomino squares between $(0,n+1)$ and $(m+1,0)$, half of $(m,0)$ and $(0,n)$ is available for a corner fold into the adjacent square closer to $(0,0)$. Make these two folds and repeat, moving closer to $(0,0)$. At last, half of cells $(0,1)$ and $(1,0)$ will be left; fold the corner as we folded the corner in the proof of Proposition \ref{LtoChain}. Repeat for the other 3 corners. \end{proof}

Notice that this alternative folding does not work for $2n \times 2m$ well-formed polyominoes of genus 1.  The key component to success in the odd-sided case is that a right triangle dissection created using the folds discussed in Proposition~\ref{propobs} must have its right-angle vertex at the midpoint on the outside edge of the center square.  This ensures that the sides of the resulting two-layer polyomino are exactly the same number of squares as the one-layer polyomino.  In a $2n \times 2m$ polyomino of genus 1, there is no center square on each side and thus the right-angle vertex of the right triangle dissection will be placed off-center, creating a two-layer polyomino with odd side-lengths.  One can fold such a polyomino, but it will be not be a {\it successful fold} in the manner stipulated.

\section*{Conclusion}

This paper provides clear algorithms to fold certain well-formed polyominoes from one level to two, including some very complicated and beautiful polyominoes.  Frederickson claimed that {\it all} well-formed polyominoes can be folded from one level to two \cite[p. 270]{fred2011}, but the systematic method for folding well-formed polyominoes of genus 0 appears to require more than corner folds and half edge cuts, and even allowing a side cut for those of genus greater than 0 is insufficient.  For instance, the authors do not have a way to fold a basic ``U shape'' polyomino as seen in Figure~\ref{Ushape} from one level to two.  

\begin{figure}[H]
\centering
\includegraphics[width=2.5in]{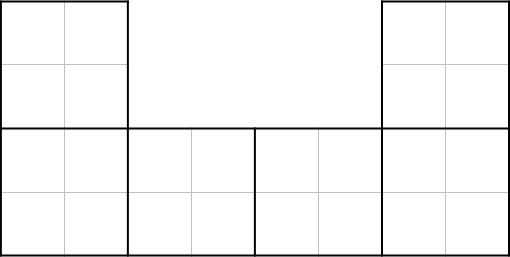}
\caption{}\label{Ushape}
\end{figure}

Furthermore, the matter of non-well-formed polyominoes is an open question.  Frederickson \cite{fred2011} explored certain families of non-well-formed polyominoes, using corner cuts (some that extend across adjacent squares) in addition to corner folds and half edge cuts.  It's straightforward, when $n$ and $k$ are positive integers, that a $(2n+1) \times (2n+1)$ square polyomino (of genus 0) can be folded from one level to two using only corner folds, but a $2n \times 2n$ square polyomino requires diagonal folds to be folded and a $(2n+1) \times (2k+1)$ rectangular polyomino cannot be folded using only corner folds. Beyond those observations, it is not clear which types of folds or cuts one be sufficient or necessary to employ in order to fold a non-well-formed polyomino.

All in all, there is still much work left if we are to understand how to fold polyominoes from one level to two.

\bibliographystyle{plain}
\bibliography{references}

\end{document}